\documentclass[12pt,english]{article}
\usepackage[T1]{fontenc}
\usepackage[latin9]{inputenc}
\usepackage{geometry}
\usepackage{babel}
\usepackage{verbatim}
\usepackage{mathtools}
\usepackage{amsmath}
\usepackage{amsthm}
\usepackage{amssymb}
\usepackage[unicode=true,pdfusetitle,
 bookmarks=true,bookmarksnumbered=false,bookmarksopen=false,
 breaklinks=false,pdfborder={0 0 1},backref=false,colorlinks=false]
 {hyperref}

\makeatletter
\numberwithin{equation}{section}
\numberwithin{figure}{section}
\theoremstyle{plain}
\newtheorem{thm}{\protect\theoremname}
\theoremstyle{definition}
\newtheorem{defn}[thm]{\protect\definitionname}
\theoremstyle{plain}
\newtheorem{conjecture}[thm]{\protect\conjecturename}
\theoremstyle{remark}
\newtheorem{rem}[thm]{\protect\remarkname}
\theoremstyle{plain}
\newtheorem{prop}[thm]{\protect\propositionname}
\theoremstyle{plain}
\newtheorem{cor}[thm]{\protect\corollaryname}
\theoremstyle{plain}
\newtheorem{lem}[thm]{\protect\lemmaname}

\date{}
\usepackage{tikz-cd}
\usepackage{wasysym}
\usepackage{MnSymbol}
\usepackage[tt=false]{libertine}
\usepackage[parfill]{parskip}
\usepackage{enumitem}
\setlist[itemize]{noitemsep,topsep=5pt}

\usepackage{titlesec}
\titleformat{\section}{\large\bfseries\filleft}{\thesection}{1em}{}[{\titlerule[0.8pt]}]

\renewcommand\labelenumi{(\roman{enumi})}
\renewcommand\theenumi\labelenumi

\DeclareMathOperator{\Spec}{Spec}

\DeclareMathOperator{\ch}{char}

\DeclareMathOperator{\Pic}{Pic}

\DeclareMathOperator{\Mor}{Mor}

\DeclareMathOperator{\Hilb}{Hilb}
\DeclareMathOperator{\ev}{ev}
\DeclareMathOperator{\Sing}{Sing}

\let\oldtheorem\thm
\renewcommand{\thm}{\oldtheorem\normalfont}
\let\oldprop\prop
\renewcommand{\prop}{\oldprop\normalfont}
\let\oldcor\cor
\renewcommand{\cor}{\oldcor\normalfont}
\let\oldlem\lem
\renewcommand{\lem}{\oldlem\normalfont}
\newenvironment{ack}{\textit{Acknowledgements.}}{}
\makeatother

\providecommand{\conjecturename}{Conjecture}
\providecommand{\corollaryname}{Corollary}
\providecommand{\definitionname}{Definition}
\providecommand{\lemmaname}{Lemma}
\providecommand{\propositionname}{Proposition}
\providecommand{\remarkname}{Remark}
\providecommand{\theoremname}{Theorem}

\begin{document}
\global\long\def\A{\mathbb{A}}%

\global\long\def\C{\mathbb{C}}%

\global\long\def\E{\mathbb{E}}%

\global\long\def\F{\mathbb{F}}%

\global\long\def\G{\mathbb{G}}%

\global\long\def\H{\mathbb{H}}%

\global\long\def\N{\mathbb{N}}%

\global\long\def\P{\mathbb{P}}%

\global\long\def\Q{\mathbb{Q}}%

\global\long\def\R{\mathbb{R}}%

\global\long\def\O{\mathcal{O}}%

\global\long\def\Z{\mathbb{Z}}%

\global\long\def\ep{\varepsilon}%

\global\long\def\laurent#1{(\!(#1)\!)}%

\global\long\def\wangle#1{\left\langle #1\right\rangle }%

\global\long\def\ol#1{\overline{#1}}%

\global\long\def\mf#1{\mathfrak{#1}}%

\global\long\def\mc#1{\mathcal{#1}}%

\global\long\def\norm#1{\left\Vert #1\right\Vert }%

\global\long\def\et{\textup{ét}}%

\global\long\def\Et{\textup{Ét}}%

\title{A converse to geometric Manin's conjecture for general low degree
hypersurfaces and Poincaré duality}
\author{Matthew Hase-Liu\thanks{Department of Mathematics, Columbia University, New York, NY}\thanks{Email address: m.hase-liu@columbia.edu}}
\maketitle
\begin{abstract}
Geometric Manin's conjecture predicts that components of the moduli
space of curves on a Fano variety parametrizing non-free curves are
pathological and arise from ``accumulating'' morphisms that increase
the Fujita invariant. By passing to positive characteristic and employing
a higher genus generalization of the circle method, we prove a converse
to this conjecture for general hypersurfaces $X$ in $\P^{n}$ of
degree $d\le n/4+3/2$, namely that there are no such accumulating
maps to $X$. One consequence of this is a version of Poincaré duality
for these moduli spaces in a range.
\end{abstract}
\tableofcontents{}

\section{Introduction}

The geometry of a Fano variety $X$ is intimately connected with the
geometry of the moduli space of curves on $X$. Mori, for instance,
used this moduli space in \cite{Mori} and introduced the technique
of bend-and-break in positive characteristic to prove Hartshorne's
conjecture that a smooth projective variety over an algebraically
closed field with ample tangent bundle is isomorphic to projective
space. 

More recently, there has been much interest in understanding well-behaved
families of curves on Fano varieties, which geometric Manin's conjecture
as stated in \cite{Lehmann_Tanimoto_2019} predicts are controlled
by certain invariants arising from the minimal model program. Specifically,
components of the moduli space that parametrize free curves are generically
smooth and are considered to be well-behaved, whereas components that
parametrize non-free curves tend to be poorly-behaved; Batryev's first
heuristic for geometric Manin's conjecture predicts that these non-free
components come from accumulating morphisms $Y\to X$, namely those
such that the Fujita invariant of $Y$ is at least as large as that
of $X$.

Ground-breaking work of Lehmann, Riedl, and Tanimoto in \cite{nonfreesections}
verified this heuristic more generally for classifying non-free sections
of $\C$-Fano fibrations in terms of the Fujita invariant. 

In this paper, we prove a converse statement to this heuristic for
general low degree hypersurfaces over $\C$, for which many aspects
of geometric Manin's conjecture are already well-understood (\cite{bilu2023motivic,cohomology_circle,BrowningSawinFree,BrowningVisheRatCurves,terminal,HarrisRothStarrRatCurvesI,haseliu2024higher,non-smooth,nonfreecurves,codimjumping,nonfreesections,Lehmann_Tanimoto_2019,cubichypersurfaces,Pugin_Thesis,RiedlYang,sawinwaring,starr2003}).
Families of rational curves on low degree smooth hypersurfaces $X$
are often irreducible and have the expected dimension, and so a reasonable
prediction is that there are in fact no accumulating maps to $X$. 

Let us briefly recall the definition of the Fujita invariant.
\begin{defn}
Let $X$ be a smooth projective variety over $\C$ and $L$ a big
and nef $\Q$-divisor on $X$. The \textit{Fujita invariant}\textbf{
}of $X$ is 
\[
a(X,L)\coloneqq\min\left\{ t\in\R\colon t[L]+\left[K_{X}\right]\text{ is pseudo-effective}\right\} .
\]
For example, if $X$ is Fano and $K_{X}$ is the canonical divisor,
then $a(X,-K_{X})=1$. 
\end{defn}

In Tanimoto's forthcoming book ``Birational Geometry and Manin's
Conjecture'', the above prediction is stated precisely as the following
conjecture.
\begin{conjecture}
[Tanimoto] Let $X\subset\P_{\C}^{n}$ be a smooth hypersurface of
degree $d$. Suppose $d\le n-2$ and $X$ is general, or $d\le n/2.$
Then, there is no proper subvariety $Y\subset X$ such that $a\left(Y,H|_{Y}\right)\ge a\left(X,H\right)$. 
\end{conjecture}

By establishing expected dimension results for spaces of high degree
maps from higher genus curves to hypersurfaces, we showed in \cite{haseliu2024higher}
that for $n$ exponentially large compared to $d$, that every smooth
hypersurface in $\P_{\C}^{n}$ of degree $d$ satisfies this conjecture.

\begin{defn}
Let $X$ be a smooth Fano variety over $\C$. An \textit{accumulating
map}\textbf{ $f\colon Y\to X$ }from a variety $Y$ is a morphism
$f$ that is generically finite, non-birational, and satisfies $a(Y,-f^{*}K_{X})\ge1$.
\end{defn}

Our main result proves many cases of Tanimoto's conjecture when $X$
is general or Fermat:
\begin{thm}
\label{thm:smallfujitageneral}Let $X\subset\P_{\C}^{n}$ be a general
or Fermat hypersurface of degree $d\ge5$ with $n\ge4d-6$. Then,
there are no accumulating maps to $X$. 
\end{thm}

\begin{rem}
We do not deal with the cases $d\le4$ because stronger statements,
i.e. for all smooth hypersurfaces, are already known in the literature
(see \cite{Lehmann_Tanimoto_2019} for example). 
\end{rem}

Combining Theorem \ref{thm:smallfujitageneral} with recent results
of Lehmann--Riedl--Tanimoto in \cite{codimjumping}, we obtain the
following corollary.
\begin{thm}
Let $X\subset\P_{\C}^{n}$ be a general or Fermat hypersurface of
degree $d\ge5$ with $n\ge4d-6$ and $C$ a smooth projective genus
$g$ curve. Then, there is some $T>0$ such that the locus of non-free
curves in $\Mor_{e}\left(C,X\right)$ has codimension at least $Te$. 

In particular, $\Mor_{e}\left(C,X\right)$ satisfies a version of
Poincaré duality in a range: Let $\ell$ be a prime number. The natural
map of étale cohomology groups with $\Q_{\ell}$-coefficients
\[
H_{c}^{2N(e)-i}\left(\Mor_{e}\left(C,X\right),\Q_{\ell}\right)\to H_{i}\left(\Mor_{e}\left(C,X\right),\Q_{\ell}\right)\left(-N(e)\right)
\]
is an isomorphism for $i<Te$, where $N(e)=\dim\Mor_{e}\left(C,X\right)=e(n+1-d)+(n-1)(1-g).$
\begin{proof}
The first claim follows immediately from Theorem \ref{thm:smallfujitageneral}
and Theorem 1.8 of \cite{codimjumping}. 

For the second, recall that for a local complete intersection $\pi\colon V\to\Spec k$
over a field $k$ with $\ch(k)\ne\ell$, the étale homology group
$H_{i}\left(V,\Q_{\ell}\right)\coloneqq H_{c}^{-i}\left(V,\pi^{!}\Q_{\ell}\right).$
Suppose $\dim V=m$ and the singular locus has dimension $s$. There
is a natural map $\pi^{*}\Q_{\ell}[m]\to\pi^{!}\Q_{\ell}[-m](-m)$
induced by the trace (see Corollary 7.7 of \cite{KW}). Let $K$ be
the cone of this map in the derived category of $\ell$-adic sheaves.
By Lemma 6.5 of \textit{loc. cit.} and the definition of perversity,
it follows that $\Q_{\ell}[m]$ and $\pi^{!}\Q_{\ell}[-m]=\mathbb{D}\left(\Q_{\ell}[m]\right)$
are perverse. $K$ is then also perverse and is supported on the singular
locus of $V$ (by Poincaré duality applied to the smooth locus). In
particular, $^{p}\mathcal{H}^{*}(\pi_{!}K)$ and hence $\mathcal{H}^{*}(\pi_{!}K)$
is zero for $*>s$. Applying $\pi_{!}$ to the triangle $\Q_{\ell}[m]\to\pi^{!}\Q_{\ell}[-m](-m)\to K\to\cdot$
and taking cohomology, we obtain the long exact sequence 
\[
\to H_{c}^{2m-i}\left(V,\Q_{\ell}\right)\to H_{c}^{-i}\left(V,\pi^{!}\Q_{\ell}(-m)\right)\to H_{c}^{m-i}\left(V,K\right)\to,
\]
so $H_{c}^{2m-i}\left(V,\Q_{\ell}\right)\cong H_{i}\left(V,\Q_{\ell}\right)(-m)$
for $m-i>s.$ The result follows by setting $V=\Mor_{e}\left(C,X\right)$
and using the first claim.
\end{proof}
\end{thm}

\begin{rem}
Completely analogous results hold for the moduli space $\mathcal{M}_{g}(X,e)$
of degree $e$ maps from smooth projective genus $g$ curves to $X$.
\end{rem}

Let us point out where the result of this paper and that of \cite{haseliu2024higher}
differ. In Theorem \ref{thm:smallfujitageneral}, we restrict ourselves
to general or Fermat hypersurfaces whereas \textit{loc. cit.} covers
arbitrary smooth hypersurfaces. On the other hand, the linear lower
bound on $n$ compared to $d$ in Theorem \ref{thm:smallfujitageneral}
is significantly better than that of the exponential lower bound in
\textit{loc. cit}. Moreover, \textit{loc. cit.} only showed the non-existence
of proper subvarieties with larger Fujita invariant, whereas Theorem
\ref{thm:smallfujitageneral} proves the stronger claim that there
are no accumulating maps at all (though this is not much harder to
prove as explained in Section \ref{sec:Reduction-from-general}).
\begin{rem}
Eric Jovinelly, Brian Lehmann, and Eric Riedl informed us that they
have a different strategy to prove a similar result, namely the non-existence
of maps $f\colon Y\to X$ such that $a\left(Y,-f^{*}K_{X}\right)>1$
for $X\subset\P_{\C}^{n}$ a general hypersurface of degree $d\le n-2$.
Note that our main theorem proves non-existence of any accumulating
map to $X$, whereas their result does not currently rule out maps
$f\colon Y\to X$ with $a\left(Y,-f^{*}K_{X}\right)=1$. On the other
hand, our degree range $d\le n/4+3/2$ is more restrictive. 
\end{rem}

To prove Theorem \ref{thm:smallfujitageneral}, we first show that
all proper subvarieties $V$ of a very general hypersurface $X$ have
strictly smaller Fujita invariant $a(V,-K_{X}|_{V})$ using the expected
dimension results for moduli spaces in large characteristic. The strategy
here is a more involved version of that in \cite{haseliu2024higher},
where we argue by contradiction and assume the existence of a proper
subvariety $V$ with large Fujita invariant. After spreading out and
passing to positive characteristic, this allows us to find a smooth
projective curve $C$ such that the space of maps from $C$ to $V$
is larger than the space of maps from $C$ to $X$, which is impossible.
The key point is that we always have lower bounds for the dimension
of the mapping space coming from deformation theory, as well as upper
bounds when the target is a very general hypersurface of low degree. 

Compared to the argument in \textit{loc. cit.}, the spreading out
procedure in the proof of Theorem \ref{thm:smallfujitageneral} is
much more delicate, using an application of the stable reduction theorem.
This requires understanding the moduli of both smooth and certain
stable curves on $X$. 

To extend this to general $X$ and all accumulating maps $Y\to X$,
we use the resolution of the Borisov--Alexeev--Borisov (BAB) conjecture.
\begin{prop}
\label{prop:reduction}Suppose for a very general hypersurface $X\subset\P_{\C}^{n}$
of degree $d\ge5$, it is known that there are no proper subvarieties
$V\subset X$ such that $a(V,-K_{X}|_{V})\ge1$. Then, for a general
hypersurface $X\subset\P_{\C}^{n}$ of degree $d$, there are no accumulating
maps to $X$. 
\end{prop}

We present this reduction from the general to very general case in
Section \ref{sec:Reduction-from-general}.

In \cite{haseliu2024higher}, the expected dimension result we use
comes from adapting the Browning--Vishe circle method strategy from
\cite{BrowningVisheRatCurves} to count tuples of global sections
satisfying the equation of the hypersurface. Many related circle method-based
arguments require the number of variables $n+1$ to be exponentially
large in $d$ because of a step called Weyl differencing, but in some
cases it is possible to remove this restriction. For example, Sawin's
work in \cite{sawinwaring} on the related Waring's problem over function
fields essentially gave an asymptotic for the number of $\F_{q}$-points
on a degree $d$ Fermat hypersurface using a similar circle method
strategy, but used Katz's bounds on exponential sums instead of Weyl
differencing, resulting in a linear lower bound on the number of variables
in terms of the degree $d$.

Let $\P^{{n+d \choose d}-1}$ represent the moduli space of degree
$d$ hypersurfaces in $\P^{n}$ and let $U_{d}\subset\P^{{n+d \choose d}-1}$
be the open locus of smooth hypersurfaces. For a stable curve $C$
with (arithmetic) genus $g$, denote by $\Pic_{C}^{e}$ for $e\ge2g-1$
the Picard stack of degree $e$ line bundles on $C$. Note that $\Pic_{C}^{e}$
is actually a scheme. Let $E\to\Pic_{C}^{e}$ be the vector bundle
whose fibers are $H^{0}\left(C,L\right)$ for $[L]\in\Pic_{C}^{e}$.
Abstractly, this is obtained by taking the universal bundle on $C\times\Pic_{C}^{e}$
and pushing forward to $\Pic_{C}^{e}$. 

Consider $\P\left(E^{n+1}\right)\times U_{d}$. Inside it, we can
consider the incidence subvariety defined by 
\[
V_{d}\coloneqq\left\{ \left(\vec{x},F\right)\in\P\left(E^{n+1}\right)\times U_{d}\colon F\left(\vec{x}\right)=0\right\} .
\]

By modifying Sawin's result in \cite{sawinwaring}, we extract the
following expected dimension statement for the generic fiber of the
family $V_{d}\to U_{d}$. 
\begin{prop}
\label{cor:generalexpdim}Let $C$ be a stable projective curve of
genus $g$ over $\F_{q}$ such that $p\coloneqq\ch\left(\F_{q}\right)>36\frac{(4d-7)d(d-2)}{d-1}$.
Let $d\ge5$ and $n\ge4d-6$ be integers. Suppose the normalizations
of the components of $C$ are $Z_{1},\ldots,Z_{t}$ over $\F_{q}$
and that $L$ is a degree $e$ line bundle on $C$ such that $\deg\left(L|_{Z_{i}}\right)$
is either 0 or at least $27(4d-7)g$. Then, the fiber above the Fermat
and the generic fiber of $V_{d}\to U_{d}$ are irreducible and have
the expected dimension $\left(n+1\right)\left(e-g+1\right)-\left(de-g+1\right)-1+g$. 
\end{prop}

We prove this proposition in Section \ref{sec:Dimension-of-moduli}. 

Next, we apply Proposition \ref{cor:generalexpdim} to prove a cleaner
version of the argument from \cite{haseliu2024higher} that applies
more generally. This should perhaps be viewed as a way to transfer
information about geometric Manin's conjecture in sufficiently positive
characteristic to Fujita invariants (over $\C$). 
\begin{prop}
\label{prop:highergenusstrat}

Let $d\ge5$ and $n\ge4d-6$ be integers. Let $X$ be a Fermat in
$\P_{\C}^{n}$ of degree $d$ or the generic fiber of the universal
family $\mathcal{H}_{d}=\left\{ \left(F,p\right)\in U_{d}\times\P^{n}\colon F(p)=0\right\} \to U_{d}$
of smooth hypersurfaces of degree $d$ in $\P_{\C}^{n}$. Then, $a\left(V,-K_{X}|_{V}\right)<1$
for any proper subvariety $V$ of $X$.
\end{prop}

We prove this proposition in Section \ref{sec:From-geometric-Manin's-1}. 

Theorem \ref{thm:smallfujitageneral} now follows easily:
\begin{proof}
[Proof of Theorem \ref{thm:smallfujitageneral}] Combine Propositions
\ref{prop:highergenusstrat} and \ref{prop:reduction}.
\end{proof}
\begin{rem}
It should be straightforward to adapt the argument of Proposition
\ref{prop:highergenusstrat} to other families of smooth Fano varieties
$\mathcal{H}\to U$ over $\C$. Moreover, the spreading out procedure
and passage to positive characteristic is only crucially needed when
handling accumulating maps $f\colon Y\to X$ for which $a\left(Y,-f^{*}K_{X}\right)=1$;
this strategy simplifies substantially when $a\left(Y,-f^{*}K_{X}\right)>1$.%
{} 
\end{rem}

\begin{rem}
Recently, Browning, Vishe, and Yamagishi in \cite{browning2024rationalcurvescompleteintersections}
developed a version of Myerson's circle method over function fields
to deduce that moduli spaces of rational curves on low degree smooth
complete intersections with marked points are irreducible of the expected
dimension. Developing a higher genus variant of this strategy (which
seems very plausible) should yield similar results to our main theorem
for complete intersections. %
\end{rem}

\begin{ack} I am grateful to my advisor Will Sawin, as well as to
Brian Lehmann, Eric Riedl, and Sho Tanimoto, for helpful comments
and suggestions. Will identified several errors and proposed streamlined
arguments; Will and Eric suggested extending the results from very
general to general hypersurfaces; and Brian and Sho explained how
to do this using the resolution of the BAB conjecture. The argument
in Section \ref{sec:Reduction-from-general} is due to Brian and Sho.

Finally, I was partially supported by National Science Foundation
Grant Number DGE-2036197.\end{ack}

\section{Reduction from general hypersurfaces to very general hypersurfaces\label{sec:Reduction-from-general}}

In this section, we prove Proposition \ref{prop:reduction}. The argument
below is due to Brian Lehmann and Sho Tanimoto. 
\begin{proof}
[Proof of Proposition \ref{prop:reduction}]

Let $X$ be a general hypersurface in $\P_{\C}^{n}$ of degree $d$.
Suppose there is a proper subvariety $Y\subset X$ with $a(Y,-K_{X}|_{Y})\ge1$.
Then, we may assume $Y$ is adjoint rigid, i.e. after replacing $Y$
with a suitable resolution of singularities, $K_{Y}-a(Y,-K_{X}|_{Y})K_{X}$
has Iitaka dimension zero, since Lemma 4.3 of \cite{geoconsistency}
shows that the canonical fibration associated to $K_{Y}-a(Y,-K_{X}|_{Y})K_{X}$
has general fiber $F$ satisfying $a(F,-K_{X}|_{F})=a(Y,-K_{X}|_{Y})$
and has corresponding Iitaka dimension zero. 

Next, Theorem 4.4 of \cite{tanimoto2018geometricaspectsmaninsconjecture}
tells us the resolution of the BAB conjecture by Birkar (see \cite{Bir1,Bir2})
implies that if $Y$ is an adjoint rigid subvariety of $X$ with $a(Y,-K_{X}|_{Y})\ge1$,
then the $-K_{X}$-degree of $Y$ is bounded by a uniform constant.
So for the universal family $\mathcal{H}_{d}\to U_{d}$ of smooth
Fano hypersurfaces in $\P_{\C}^{n}$ of degree $d$, there is a finite
type subscheme parametrizing $Y$ in the relative Hilbert scheme $\Hilb$$\left(\mathcal{H}_{d}/U_{d}\right)$. 

The image of such families in $U_{d}$ is a constructible set, say
$S$. By assumption, since we know the claim about proper subvarieties
of very general hypersurfaces in $\P_{\C}^{n}$ of degree $d$ holds,
it follows that a very general point of $U_{d}$ is not contained
in $S$. 

This implies that $S$ is contained in a proper closed subset, which
is equivalent to saying that for a general hypersurface $X$, all
proper subvarieties have smaller Fujita invariant.

To extend this from subvarieties to accumulating maps, suppose for
contradiction that we have an accumulating map $f\colon Y\to X$. 

If $\dim Y<\dim X$, then $a(Y,-f^{*}K_{X})\le a\left(f(Y),-K_{X}|_{f(Y)}\right)$
by Lemma 2.4 of \cite{senguptaFujita}, which is moreover less than
one by the result for proper subvarieties. So we may assume $f$ is
dominant. 

Arguing as we did before, we may assume $Y$ is adjoint rigid. But
Theorem 1.3 of \cite{senguptaFujita} shows that for such a map $f$,
any component $B$ of the branch locus of the map satisfies $a(B,-K_{X}|_{B})>1$,
which is impossible again by the result for proper subvarieties. By
Tag 0EA4 of \cite{key-2}, the branch locus is necessarily a divisor,
so this implies that $f$ is étale in codimension one. Restricting
$f$ then gives a finite étale cover in codimension one, which by
Grothendieck's version of Zariski--Nagata purity (see Tag 0BMA of
\cite{key-2}) extends to a non-trivial finite étale cover of $X$.
On the other hand, $X$ is a smooth hypersurface in $\P_{\C}^{n}$,
so it is simply connected. This means that the étale fundamental group
of $X$ is actually  trivial, giving a contradiction.
\end{proof}

\section{Dimension of moduli spaces of curves on very general hypersurfaces\label{sec:Dimension-of-moduli}}

In this section, we prove Proposition \ref{cor:generalexpdim}. 

In \cite{sawinwaring}, Sawin obtained an asymptotic for Waring's
problem in the function field setting for any smooth projective curve
over a finite field of sufficiently large characteristic. By modifying
the proofs in \textit{loc. cit.} to include congruence conditions,
we obtain the following generalization that is uniform in $q$.
\begin{thm}
\label{thm:Willgeneralization}Let $C$ be a smooth projective curve
of genus $g$ over $\F_{q}$, $d\ge5$ and $n\ge4d-6$ integers, $B$
an effective divisor on $C$ of degree $b$, $\vec{P}$ a fixed tuple
in $H^{0}\left(B,L|_{B}\right)^{n+1}$, $L$ a degree $e\ge b+27(4d-7)g$
line bundle, $F=x_{0}^{d}+\cdots+x_{n}^{d}$, and $p\coloneqq\ch\left(\F_{q}\right)>36\frac{(4d-7)d(d-2)}{d-1}$.
Then, 
\[
\frac{\#\left\{ \vec{x}\in H^{0}\left(C,L\right)^{n+1}\colon F\left(\vec{x}\right)=0,\vec{x}|_{B}=\vec{P}\right\} }{q^{e\left(n+1-d\right)+n(1-g)-bn}}\to1
\]
 as $q\to\infty$. 
\end{thm}

\begin{rem}
It is possible to improve the lower bounds on $e$ and $p$, but doing
so is not relevant for applications to the Fujita invariant. 
\end{rem}

\begin{proof}
[Proof sketch] Let us briefly recall the overall outline of Sawin's
argument in \cite{sawinwaring}, which begins using the usual approach
via the circle method by expressing the counting question as a sum
of exponential sums and partitioning these into major and minor arcs.
His treatment of major arcs is similar to that of previous work, such
as in \cite{BrowningSawinFree,BrowningVisheRatCurves,haseliu2024higher},
but his strategy for the minor arcs involves viewing them as complete
exponential sums over finite fields and then applying Katz's bounds
in terms of the dimensions of certain singular loci. Sawin bounds
the singular loci by expressing them in terms of pairs $(a,c)$ of
sections of line bundles, stratifying them according to the multiplicities
of their roots, and then bounding the strata by the dimensions of
their tangent spaces. 

Adding congruence conditions---imposing that the tuples $\vec{x}\in H^{0}\left(C,L\right)^{n+1}$
are determined when restricted to an effective divisor $B\subset C$---requires
only a small tweaking to \cite{sawinwaring}, but it does not seem
possible to conclude Theorem \ref{thm:Willgeneralization} directly
from Theorem 3.15 of \textit{loc. cit}. In particular, a tuple $\vec{x}\in H^{0}\left(C,L\right)^{n+1}$
satisfying $\vec{x}|_{B}=\vec{P}$ is an element of $\vec{P}+H^{0}\left(C,L(-B)\right)^{n+1}$,
where by abuse of notation we pick $\vec{P}$ to be some lift to $H^{0}\left(C,L\right)^{n+1}$.
Then, instead of requiring that $\deg L\ge2g-1$ as in \textit{loc.
cit.}, we require $\deg L(-B)\ge2g-1$, which gives the condition
$e\ge b+2g-1$. In particular, this implies that $L(-B)$ is a non-special
line bundle. 

To express $\#\left\{ \vec{x}\in H^{0}\left(C,L\right)^{n+1}\colon F\left(\vec{x}\right)=0,\vec{x}|_{B}=\vec{P}\right\} $
as a sum of exponential sums, we express the counting problem as a
sum over the ``circle'' $H^{0}\left(C,L^{\otimes d}(-B)\right)^{\vee}$
(instead of $H^{0}\left(C,L^{\otimes d}\right)^{\vee}$ in \textit{loc.
cit.}), where in \textit{loc. cit.} Sawin writes $k$ for $d$ and
$s$ for $n+1$. We define 
\[
S_{1}(\alpha)_{i}=\sum_{a\in H^{0}\left(C,L(-B)\right)}\psi\left(\alpha\left(\left(a+P_{i}\right)^{d}\right)\right),
\]
where $\psi$ is some-trivial additive character $\F_{q}\to\C^{\times}$.

Lemma 2.1 of \textit{loc. cit.} (which we only need for $f=0$) should
be replaced in our setting with 
\[
\#\left\{ \vec{x}\in\vec{P}+H^{0}\left(C,L(-B)\right)^{n+1}\colon F\left(\vec{x}\right)=0\right\} =\frac{1}{q^{de+1-g-b}}\sum_{\alpha\in H^{0}\left(C,L^{\otimes d}(-B)\right)^{\vee}}S_{1}(\alpha)_{0}\cdots S_{1}(\alpha)_{n}.
\]
We say $\alpha\in H^{0}\left(C,L^{\otimes d}(-B)\right)^{\vee}$ factors
through an effective divisor $Z\subset C$ or $\alpha\sim Z$ more
concisely if the map $H^{0}\left(C,L^{\otimes d}(-B)\right)\to\F_{q}$
factors through the restriction $H^{0}\left(C,L^{\otimes d}(-B)\right)\to H^{0}\left(Z,L^{\otimes d}(-B)\right)$.
Let $\deg\alpha$ be the minimum integer $\deg Z$ such that $\alpha\sim Z$.
Lemma 9 of \cite{haseliu2024higher} ensures that $\deg\alpha\le(de-b)/2+1$.
When $\deg\alpha\le e-b-2g+1$, we say $\alpha$ is a \textit{major}
arc; else, $\alpha$ is a \textit{minor} arc.

Since we only need to understand the leading term of the asymptotic
uniformly in $q$ but not in $e$ (note that \cite{sawinwaring} addresses
uniformity in $e$ as well), we use a slightly simpler (but essentially
the same kind of) argument to handle the major arcs as described in
Section 4 of \cite{haseliu2024higher}. 

To show 
\[
\frac{\#\left\{ \vec{x}\in H^{0}\left(C,L\right)^{n+1}\colon F\left(\vec{x}\right)=0,\vec{x}|_{B}=\vec{P}\right\} }{q^{e\left(n+1-d\right)+n(1-g)-bn}}\to1,
\]
it suffices to show 
\[
\frac{1}{q^{(e-g+1-b)(n+1)}}\sum_{\alpha\in H^{0}\left(C,L^{\otimes d}(-B)\right)^{\vee}}S_{1}(\alpha)_{0}\cdots S_{1}(\alpha)_{n}\to1.
\]
Replacing $e$ with $e-b$ in Lemma 11, combined with Corollary 12,
Lemma 13, and Corollary 14 of \cite{haseliu2024higher} gives the
major arc contribution: 
\[
\frac{1}{q^{(e-g+1-b)(n+1)}}\sum_{\deg\alpha\le e-b-2g+1}S_{1}(\alpha)_{0}\cdots S_{1}(\alpha)_{n}\to\frac{S_{1}(0)_{0}\cdots S_{1}(0)_{n}}{q^{(e-g+1-b)(n+1)}}=\frac{q^{(n+1)(e-g+1-b)}}{q^{(e-g+1-b)(n+1)}}=1.
\]

It remains to verify the minor arc contribution, namely that 

\begin{equation}
\frac{1}{q^{(e-g+1-b)(n+1)}}\sum_{\deg\alpha>e-b-2g+1}S_{1}(\alpha)_{0}\cdots S_{1}(\alpha)_{n}\to0.\label{eq:minorarc}
\end{equation}
Returning to Sawin's argument, let us define the relevant singular
locus that we wish to bound. The idea is to view $H^{0}\left(C,L(-B)\right)$
as an affine space---its $R$-points are $H^{0}\left(C,L(-B)\right)\otimes_{\F_{q}}R$
for any $\F_{q}$-algebra $R$. Fixing $\alpha\in H^{0}\left(C,L^{\otimes d}(-B)\right)^{\vee}$,
let 
\[
\Sing_{\alpha}=\left\{ a\in H^{0}\left(C,L(-B)\right)\colon\alpha\left(a^{d-1}b\right)\text{ for all }b\in H^{0}\left(C,L(-B)\right)\right\} .
\]
Note that we are viewing $a^{d-1}$ as an element of $H^{0}\left(C,L^{\otimes d}\right)$
using the inclusion $H^{0}\left(C,L(-B)^{\otimes d}\right)\hookrightarrow H^{0}\left(C,L^{\otimes d}\right)$.

Lemma 3.1 of \cite{sawinwaring} should be replaced with the claim
that for any $i$,
\[
\left|S_{1}(\alpha)_{i}\right|\le3\left(d+1\right)^{e-b-g+1}q^{\frac{e-b+g-1+\dim\Sing_{\alpha}}{2}}.
\]
To see this, $X$ should be replaced with $\P^{e-b+1-g}$ with variables
$X_{0},\ldots,X_{e-b-g}$ defined by a single linear form inside $\P^{e-b+2-g}$,
$L$ should be the linear form $\{X_{0}=0\}$ so that the complement
is an affine space identified with $H^{0}\left(C,L(-B)\right)$, and
$H$ is the homogenization with respect to $X_{0}$ of the polynomial
function $\alpha\left((a+P_{i})^{d}\right)$ on $H^{0}\left(C,L(-B)\right)$.
Then, the singular locus in consideration is that of $X\cap L\cap H$,
which corresponds to the hypersurface in $\P^{e-b-g}$ defined by
the leading term of $\alpha\left((a+P_{i})^{d}\right)$, which is
$\alpha\left(a^{d}\right)$. The rest of the proof of Lemma 3.1 of
\textit{loc. cit.} then goes through identically.

Denote by $\overline{\alpha}\in H^{0}\left(Z,L^{\otimes d}(-B)\right)^{\vee}$
the restriction of $\alpha$ to $Z$. Lemma 3.2 of \textit{loc. cit.}
is the same except $L^{k}$ should be replaced with $L^{\otimes d}(-B)$,
and the proof is identical: Each $\overline{\alpha}$ has a unique
$\widetilde{\alpha}\in H^{0}\left(Z,K_{C}(Z)\otimes L^{\otimes-d}(B)\right)$
such that $\overline{\alpha}(f)$ is given by the sum over points
$v\in Z$ of the residue of $f\widetilde{\alpha}\in H^{0}\left(Z,K_{C}(-Z)\right)$
at $v$. 

Noting that the bounds for $|S_{1}(\alpha)_{i}|$ are independent
of $i$, Lemma 3.14 of \textit{loc. cit.}, which combines the previous
lemmas, as well as Lemmas 3.3 through 3.13 of \textit{loc. cit.} applied
to $L^{k}$ replaced with $L^{\otimes d}(-B)$ and $L$ replaced with
$L(-B)$, to bound the minor arc contribution, should then be replaced
with the statement that for all $0<\delta<\frac{\frac{n+1}{2}-d}{d-1}$,
we have 
\begin{align}
 & \sum_{\deg\alpha>e-b-2g+1}|S_{1}(\alpha)_{0}\cdots S_{1}(\alpha)_{n}|\nonumber \\
 & \le dq^{(d+1)(e-b)+2-2g}3^{n-1}\left(d+1\right)^{(n-1)(e-b+1-g)}q^{(n-1)\frac{e-b+2+\left(e-b+\frac{2\max(2g-1,0)}{d-2}\right)\gamma_{d,p}}{2}}\nonumber \\
 & \qquad+O_{n,d,\delta}\left((d+1)^{(n+1)(e-b+1-g)}q^{(n+1)(e-b+1-g)}\left(1+q^{-1/2}\right)^{O_{d}(g)}q^{-\delta(e-b-2g+2)}\right),\label{eq:replacement3.14}
\end{align}
where $\gamma_{d,p}$ is a certain number defined in \textit{loc.
cit.}, whose only property we will use is that 
\begin{equation}
\frac{d-2}{2d-2}\le\gamma_{d,p}\le\frac{d-2}{2d-2}\left(1+\frac{d}{p}\right)\le\frac{d-2}{d-1}\label{eq:boundsongamma}
\end{equation}
by Lemmas 3.8 and 3.9 of \textit{loc. cit.}

Recall that it remains to verify the minor arc contribution (\ref{eq:minorarc})
\[
\frac{1}{q^{(e-g+1-b)(n+1)}}\sum_{\deg\alpha>e-b-2g+1}S_{1}(\alpha)_{0}\cdots S_{1}(\alpha)_{n}\to0.
\]
Using our replacement (\ref{eq:replacement3.14}) of Lemma 3.14 of
\textit{loc. cit.} and dropping the terms that do not depend on $q$,
we obtain 
\begin{align*}
 & \frac{1}{q^{(e-g+1-b)(n+1)}}\sum_{\deg\alpha>e-b-2g+1}|S_{1}(\alpha)_{0}\cdots S_{1}(\alpha)_{n}|\\
 & \ll_{q}\frac{q^{(d+1)(e-b)+2-2g+(n-1)\frac{e-b+2+\left(e-b+\frac{2\max(2g-1,0)}{d-2}\right)\gamma_{d,p}}{2}}+q^{(n+1)(e-b+1-g)}\left(1+q^{-1/2}\right)^{O_{d}(g)}q^{-\delta(e-b-2g+2)}}{q^{(e-b-g+1)(n+1)}}\\
 & \ll_{q}q^{(d+1)(e-b)+2-2g+(n-1)\frac{e-b+2+\left(e-b+\frac{2\max(2g-1,0)}{d-2}\right)\gamma_{d,p}}{2}-(e-b-g+1)(n+1)},
\end{align*}
so it remains to verify that 
\[
(d+1)(e-b)+2-2g+(n-1)\frac{e-b+2+\left(e-b+\frac{2\max(2g-1,0)}{d-2}\right)\gamma_{d,p}}{2}-(e-b-g+1)(n+1)<0.
\]

Rearranging, this is the same as checking that 
\begin{equation}
n>\frac{(2d-1-\gamma_{p,d})(e-b)-2g-\frac{2\gamma_{p,d}\max\left(2g-1,0\right)}{d-2}}{(1-\gamma_{p,d})(e-b)-2g-\frac{2\gamma_{p,d}\max\left(2g-1,0\right)}{d-2}}.\label{eq:nbound}
\end{equation}
For $p>36\frac{(4d-7)d(d-2)}{d-1}$, using (\ref{eq:boundsongamma}),
it is easy to verify that $2d-5-(4d-7)\gamma_{p,d}>1/9$, so 
\[
e-b\ge27(4d-7)g=\frac{(4d-7)\left(2g+2\frac{2g}{d-1}\right)}{1/9}>\frac{(4d-7)\left(2g+2\frac{\max\left(2g-1,0\right)}{d-2}\gamma_{d,p}\right)}{2d-5-(4d-7)\gamma_{p,d}},
\]
which after plugging into the right-hand side of (\ref{eq:nbound})
gives
\[
n\ge4d-6>\frac{(2d-1-\gamma_{p,d})(e-b)-2g-\frac{2\gamma_{p,d}\max\left(2g-1,0\right)}{d-2}}{(1-\gamma_{p,d})(e-b)-2g-\frac{2\gamma_{p,d}\max\left(2g-1,0\right)}{d-2}},
\]
as desired.
\end{proof}
\begin{cor}
\label{cor:Willstable}Let $C$ be a stable projective curve of genus
$g$ over $\F_{q}$ whose normalizations of components are $Z_{1},\ldots,Z_{t}$
over $\F_{q}$, $d\ge5$ and $n\ge4d-6$ integers, $L$ a line bundle
on $C$ of degree $e$ such that $\deg\left(L|_{Z_{i}}\right)\ge27(4d-7)g$,
$F=x_{0}^{d}+\cdots+x_{n}^{d}$, and $p\coloneqq\ch\left(\F_{q}\right)>36\frac{(4d-7)d(d-2)}{d-1}$.
Then, 
\[
\frac{\#\left\{ \vec{x}\in H^{0}\left(C,L\right)^{n+1}\colon F\left(\vec{x}\right)=0\right\} }{q^{e\left(n+1-d\right)+n(1-g)}}\to1
\]
 as $q\to\infty$. 
\begin{proof}
Each $Z_{i}$ is a smooth projective curve with genus $g\left(Z_{i}\right)$.
We proceed by induction on $t$. 

Let us start with the base case $t=1$. Denote by $\Gamma$ the stable
graph of $C$, which has a single vertex. When there are no edges,
we conclude by Theorem \ref{thm:Willgeneralization} with $b=0$.
For $t=1$ with (self) edges in $\Gamma$, note that the points on
$Z_{1}$ that are to be identified with each other come in distinct
pairs, say $\{p_{1},p_{1}'\},\ldots,\{p_{a},p_{a}'\}$. 

Note that a map from $C$ to $X$ is the same as a map from $Z_{1}$
to $X$ so that $p_{i}$ and $p_{i}'$ map to the same point in $X$.
In our setting, however, note that we don't exactly work with maps
but rather tuples of global sections satisfying the equation $F$.

Let $B$ be the divisor associated to $p_{1}'\cup\cdots\cup p_{a}'$
and suppose $\deg B=b$. Then, 
\begin{align*}
\left\{ \vec{x}\in H^{0}\left(C,L\right)^{n+1}\colon F\left(\vec{x}\right)=0\right\}  & =\left\{ \vec{x}\in H^{0}\left(Z_{1},L\right)^{n+1}\colon F\left(\vec{x}\right)=0,\vec{x}|_{B}\text{ fixed}\right\} 
\end{align*}
(more precisely, $p_{i}'$ is mapped to wherever $p_{i}$ is mapped
to). 

Since $\deg\left(L|_{Z_{1}}\right)\ge27(4d-7)g=27(4d-7)\left(g\left(Z_{1}\right)+1+b-1\right)\ge b+27(4d-7)g\left(Z_{1}\right)$,
we may apply Theorem \ref{thm:Willgeneralization} to conclude. 

Suppose then that the conclusion of the theorem is true for stable
curves with at most $t-1$ components. 

We can write $C$ as the union of $C'$ and $Z'_{t}$, where $C'$
is a stable curve with at most $t-1$ components and $Z_{t}'$ is
at worst nodal with $Z_{t}$ the normalization of $Z_{t}'$. More
precisely, $C$ is the pushout of $C'$ and $Z'_{t}$ along their
intersection $W=p_{1}\cup\cdots\cup p_{n_{t}}.$ Then, any morphism
$C\to X$ comprises the data of a morphism $C'\to X$ and a morphism
$Z_{t}'\to X$ that agree when restricted to $W$. Consider the restriction
maps $\ev\colon H^{0}\left(C',L\right)^{n+1}\to H^{0}\left(W,L\right)^{n+1}$
and $H^{0}\left(Z_{t}',L\right)^{n+1}\to H^{0}\left(W,L\right)^{n+1}$,
which is analogous to the evaluation maps $\ev\colon\Mor\left(C',X\right)\to X^{n_{t}}$
and $\Mor\left(Z_{t}',X\right)\to X^{n_{t}}$ sending a morphism $C'\to X$
(resp. a morphism $W\to X$) to the images in $X$ of the closed points
$p_{i}$. Then, 
\begin{align*}
 & \left\{ \vec{x}\in H^{0}\left(C,L\right)^{n+1}\colon F\left(\vec{x}\right)=0\right\} \\
 & =\left\{ \vec{x}\in H^{0}\left(C',L\right)^{n+1}\colon F\left(\vec{x}\right)=0\right\} \times_{\left\{ \vec{x}\in H^{0}\left(W,L\right)^{n+1}\colon F\left(\vec{x}\right)=0\right\} }\left\{ \vec{x}\in H^{0}\left(Z_{t}',L\right)^{n+1}\colon F\left(\vec{x}\right)=0\right\} .
\end{align*}
In $\Gamma$, the vertex $Z_{t}'$ has $n_{t}$ edges that are not
self edges. Suppose $Z_{t}'$ has $a$ self edges coming in pairs
$\{p_{1},p_{1}'\},\ldots,\{p_{a},p_{a}'\}$, and set $B$ to be the
divisor associated to $p_{1}'\cup\cdots\cup p_{a}'$. Let $\deg W=w$
and $\deg B=b$. Then, 
\[
\deg\left(L|_{Z_{t}}\right)\ge27(4d-7)g\ge27(4d-7)\left(g\left(C'\right)+g\left(Z_{t}'\right)+w-1\right)\ge w+b+27(4d-7)g\left(Z_{t}\right),
\]

so Theorem \ref{thm:Willgeneralization} tells us the fibers of $\left\{ \vec{x}\in H^{0}\left(Z_{t}',L\right)^{n+1}\colon F\left(\vec{x}\right)=0\right\} \to\left\{ \vec{x}\in H^{0}\left(W,L\right)^{n+1}\colon F\left(\vec{x}\right)=0\right\} $,
which by the same argument in the $t=0$ case is simply $\left\{ \vec{x}\in H^{0}\left(Z_{t},L\right)^{n+1}\colon F\left(\vec{x}\right)=0,\vec{x}|_{W\cup B}\text{ fixed}\right\} $,
are geometrically irreducible of the expected dimension. Then, by
miracle flatness, since the source of the evaluation map is a local
complete intersection, the target is smooth, and the fibers are equidimensional,
we have that the evaluation map $\left\{ \vec{x}\in H^{0}\left(Z_{t}',L\right)^{n+1}\colon F\left(\vec{x}\right)=0\right\} \to\left\{ \vec{x}\in H^{0}\left(W,L\right)^{n+1}\colon F\left(\vec{x}\right)=0\right\} $
is flat with geometrically irreducible fibers. We conclude that $\left\{ \vec{x}\in H^{0}\left(C,L\right)^{n+1}\colon F\left(\vec{x}\right)=0\right\} $
is geometrically irreducible. 

Let $e_{i}=\deg\left(L|_{Z_{i}}\right)$ and write $e_{C'}=e_{1}+\cdots+e_{t-1}$.
Note that $e=e_{C'}+e_{t}$. We have $g=g\left(C'\right)+g\left(Z_{t}'\right)+w-1.$
By the induction hypothesis and flatness, we also have 
\begin{align*}
 & \dim\left\{ \vec{x}\in H^{0}\left(C,L\right)^{n+1}\colon F\left(\vec{x}\right)=0\right\} \\
 & =\dim\left\{ \vec{x}\in H^{0}\left(C',L\right)^{n+1}\colon F\left(\vec{x}\right)=0\right\} +\dim\left\{ \vec{x}\in H^{0}\left(Z_{t}',L\right)^{n+1}\colon F\left(\vec{x}\right)=0\right\} \\
 & \qquad-\dim\left\{ \vec{x}\in H^{0}\left(W,L\right)^{n+1}\colon F\left(\vec{x}\right)=0\right\} \\
 & =\left(n+1\right)\left(e_{C'}-g\left(C'\right)+1\right)-\left(de_{C'}-g\left(C'\right)+1\right)+\left(n+1\right)\left(e_{t}-g\left(Z_{t}'\right)+1\right)-\left(de_{t}-g\left(Z_{t}'\right)+1\right)\\
 & \qquad-wn\\
 & =\left(n+1\right)\left(e-g\left(C'\right)-g\left(Z_{t}'\right)+2\right)-\left(de-g\left(C'\right)-g\left(Z_{t}'\right)+2\right)-wn\\
 & =\left(n+1\right)\left(e-g+1+w\right)-\left(de-g+1+w\right)-wn\\
 & =\left(n+1\right)\left(e-g+1\right)-\left(de-g+1\right).
\end{align*}
Since $\left\{ \vec{x}\in H^{0}\left(C,L\right)^{n+1}\colon F\left(\vec{x}\right)=0\right\} $
is geometrically irreducible of the expected dimension, the Lang--Weil
bounds give the desired result.
\end{proof}
\end{cor}

\begin{proof}
[Proof of Proposition \ref{cor:generalexpdim}] By general lower bounds
on dimension from deformation theory (as explained for example in
Theorem 1.2 of \cite{Kollar_Rational_Curves}), it suffices to prove
an upper bound on dimension that matches the lower bound. 

Note that $V_{d}\to U_{d}$ is proper and one of its fibers (above
the Fermat hypersurface) has dimension exactly $\mu\coloneqq\left(n+1\right)\left(e-g+1\right)-\left(de-g+1\right)-1+g$
by Corollary \ref{cor:Willstable} after possibly base-extending and
applying the Lang--Weil bounds, so upper semicontinuity of fiber
dimension implies the generic fiber also has dimension $\mu$. 
\end{proof}

\section{From geometric Manin's conjecture in positive characteristic to Fujita
invariants\label{sec:From-geometric-Manin's-1}}

In this section, we prove Proposition \ref{prop:highergenusstrat}
\begin{proof}
[Proof of Proposition \ref{prop:highergenusstrat}]

For the sake of contradiction, let $V$ be a proper subvariety of
$X$ such that $a\left(V,-K_{X}|_{V}\right)\ge1$. 

Now, suppose $Y\to V$ is a resolution of singularities. Then, $a\left(Y,-K_{X}|_{Y}\right)=a\left(V,-K_{X}|_{V}\right)\ge1$.
This implies that $K_{Y}-K_{X}|_{Y}+\frac{1}{m}K_{X}|_{Y}$ is not
pseudo-effective for any integer $m>0$. By the duality of the cone
of pseudo-effective divisors and the closure of the cone of movable
curves as established by \cite{BDPP}, it follows that for a general
point $y\in Y$, we have a smooth projective curve $C_{m}$ passing
through $y$ such that 
\[
\left(-K_{Y}+K_{X}|_{Y}-\frac{1}{m}K_{X}|_{Y}\right)\cdot C_{m}>0,
\]
i.e. 
\begin{align*}
\left(-K_{Y}+K_{X}|_{Y}\right)\cdot C_{m} & >\frac{1}{m}K_{X}|_{Y}\cdot C_{m}.
\end{align*}

Choose an integer $m\ge1$ such that 

\begin{equation}
m>A\coloneqq27(4d-7)\label{eq:mchoice}
\end{equation}

Let $C=C_{m}$ and write 
\[
e=K_{X}|_{Y}\cdot C.
\]

We now pass to positive characteristic by spreading out. Let us describe
this procedure in detail: 

Let $\Spec K\left(U_{d}\right)$ be the generic point of $U_{d}$.
We have a diagram \[\begin{tikzcd} 	X & {\P^n_{K(U_d)}} \\ 	{\Spec K(U_d)} 	\arrow[hook, from=1-1, to=1-2] 	\arrow[from=1-1, to=2-1] 	\arrow[from=1-2, to=2-1] \end{tikzcd}\]that
can be spread out to \[\begin{tikzcd} 	{\mathcal{X}_{U_d}} & {\P^n_{R}} \\ 	{\Spec R} 	\arrow[hook, from=1-1, to=1-2] 	\arrow[from=1-1, to=2-1] 	\arrow[from=1-2, to=2-1] \end{tikzcd}\]

for a sufficiently small affine open $\Spec R\hookrightarrow U_{d}$,
i.e. to a family of smooth hypersurfaces $\mathcal{X}_{U_{d}}\to\Spec R$
so that the generic fiber of this family is $X$ and $\mathcal{X}_{U_{d}}=\Spec R\times_{U_{d}}\mathcal{H}_{d}$
(by the universal property of the moduli space of degree $d$ hypersurfaces
in $\P^{n}$). Since $Y\to V\hookrightarrow X$ and $C$ are all of
finite type over the generic point of $U_{d}$, by possibly shrinking
$R$, we may moreover assume that $C\hookrightarrow Y\to V\hookrightarrow X$
spreads out to \[\begin{tikzcd} 	{\mathcal{Y}_{U_d}} & {\mathcal{V}_{U_d}} & {\mathcal{X}_{U_d}} & {\P^n_{R}} \\ 	& {\mathcal{C}_{U_d}} & {\Spec R} 	\arrow[from=1-1, to=1-2] 	\arrow[from=1-1, to=2-3] 	\arrow[hook, from=1-2, to=1-3] 	\arrow[from=1-2, to=2-3] 	\arrow[hook, from=1-3, to=1-4] 	\arrow[from=1-3, to=2-3] 	\arrow[from=1-4, to=2-3] 	\arrow[hook, from=2-2, to=1-1] 	\arrow[from=2-2, to=2-3] \end{tikzcd}\]
i.e., where $Y$ is the generic fiber of a smooth and projective family
$\mathcal{Y}_{U_{d}}\to\Spec R,$ $C$ is the generic fiber of a smooth
and projective family $\mathcal{C}_{U_{d}}\to\Spec R$, and $V$ is
the generic fiber of a flat and projective family $\mathcal{V}_{U_{d}}\to\Spec R$. 

Since $\Spec R$ and $U_{d}$ are both of finite type over $\Spec\C$,
we may spread out both to an affine open inclusion $\Spec\mathcal{R}\hookrightarrow\mathcal{U}_{d},$
living over $\Spec\Lambda$, where $\Lambda$ is a finitely-generated
$\Z$-algebra that is an integral domain (so that $\Spec\Lambda$
is irreducible) obtained by adjoining coefficients defining $\Spec R$
and $U_{d}$. More precisely, the induced map from the generic fiber
of $\Spec\mathcal{R}\to\Spec\Lambda$ to the generic fiber of $\mathcal{U}_{d}\to\Spec\Lambda$,
base changed to $\C$, is the affine open immersion $\Spec R\hookrightarrow U_{d}$. 

We have the following commutative diagram, where the rightmost square
is Cartesian.\[\begin{tikzcd} 	{\mathcal{Y}_{U_{d}}} & {\mathcal{V}_{U_{d}}} & {\mathcal{X}_{U_{d}}} & \mathcal{H}_{d} \\ 	& {\mathcal{C}_{U_{d}}} & {\Spec{R}} & {U_{d}} 	\arrow[from=1-1, to=1-2] 	\arrow[from=1-1, to=2-3] 	\arrow[hook, from=1-2, to=1-3] 	\arrow[from=1-2, to=2-3] 	\arrow[from=1-3, to=1-4] 	\arrow[from=1-3, to=2-3] 	\arrow[from=1-4, to=2-4] 	\arrow[hook, from=2-2, to=1-1] 	\arrow[from=2-2, to=2-3] 	\arrow[hook, from=2-3, to=2-4] \end{tikzcd}\]

By possibly shrinking $\Lambda$, we may assume that this diagram
(which lives over $\C$) spreads out to the diagram \[\begin{tikzcd} 	{\mathcal{Y}_\mathcal{W}} & {\mathcal{V}_\mathcal{W}} & {\mathcal{X}_\mathcal{W}} & {\mathcal{F}} \\ 	& {\mathcal{C}_\mathcal{W}} & {\Spec{\mathcal{R}}} & {\mathcal{W}} 	\arrow[from=1-1, to=1-2] 	\arrow[from=1-1, to=2-3] 	\arrow[hook, from=1-2, to=1-3] 	\arrow[from=1-2, to=2-3] 	\arrow[from=1-3, to=1-4] 	\arrow[from=1-3, to=2-3] 	\arrow[from=1-4, to=2-4] 	\arrow[hook, from=2-2, to=1-1] 	\arrow[from=2-2, to=2-3] 	\arrow[hook, from=2-3, to=2-4] \end{tikzcd}\]which
again means that after taking the same diagram after taking the generic
fibers (of each object mapping to $\Spec\Lambda$), and then base
changing to $\C$, we obtain the previous diagram.

For $t\in\Spec\mathcal{R}$, the zero-cycles $\left(-K_{\left(\mathcal{Y}_{\mathcal{W}}\right)_{t}}+K_{\left(\mathcal{X}_{\mathcal{W}}\right)_{t}}|_{\left(\mathcal{Y}_{\mathcal{W}}\right)_{t}}\right)\cdot\left(\mathcal{C}_{\mathcal{W}}\right)_{t}$
have constant degree, and hence all have degree $>\frac{1}{m}e,$
since over the generic point of the generic fiber of $\Spec\mathcal{R}\to\Spec\Lambda$,
this is true.

By Zariski's lemma, every closed point of $\Spec\Lambda$ has finite
residue field. Also, since the generic fiber of the map $\mathcal{W}\to\Spec\Lambda$
is irreducible, there is an open neighborhood of $\Spec\Lambda$ such
that the fiber above any point is irreducible. Take such a point $P$
and suppose it has residue field $\F_{q}$. We may moreover assume
the characteristic $p$ of $\F_{q}$ is large enough so that
\begin{equation}
\left(\frac{1}{A}-\frac{1}{m}\right)p^{2}>p+\frac{p-1}{2}\text{ and }p>36\frac{(4d-7)d(d-2)}{d-1}.\label{eq:pchoice}
\end{equation}

Let $\mathcal{W}_{\F_{q}}$ be $\mathcal{W}\otimes_{\Lambda}\kappa(P)$,
which is isomorphic to the moduli space of smooth hypersurfaces of
degree $d$ in $\P_{\F_{q}}^{n}.$ Let $\mathcal{F}_{\F_{q}}$ be
$\mathcal{F}\times_{\mathcal{W}}\mathcal{W}_{\F_{q}}$, $\mathcal{X}_{\mathcal{W}_{\F_{q}}}$
be $\mathcal{\mathcal{X}}_{\mathcal{W}}\times_{\mathcal{W}}\mathcal{W}_{\F_{q}}$,
and similarly for $\mathcal{R}_{\F_{q}},\mathcal{Y}_{\mathcal{W}_{\F_{q}}},\mathcal{V}_{\mathcal{W}_{\F_{q}}},$
and $\mathcal{C}_{\mathcal{W}_{\F_{q}}}$. Taking the open neighborhood
in the previous paragraph to be small enough, we may moreover assume
$\mathcal{Y}_{\mathcal{W}_{\F_{q}}}$ and $\mathcal{C}_{\mathcal{W}_{\F_{q}}}$
are smooth over $\Spec\mathcal{R}_{\F_{q}}$ and that $\mathcal{V}_{\mathcal{W}_{\F_{q}}}$
is flat over $\Spec\mathcal{R}_{\F_{q}}$.

We have the diagram\[\begin{tikzcd} 	{\mathcal{Y}_{\mathcal{W}_{\F_q}}} & {\mathcal{V}_{\mathcal{W}_{\F_q}}} & {\mathcal{X}_{\mathcal{W}_{\F_q}}} & {\mathcal{F}_{\F_q}} \\ 	& {\mathcal{C}_{\mathcal{W}_{\F_q}}} & {\Spec{\mathcal{R}_{\F_q}}} & {\mathcal{W}_{\F_q}} 	\arrow[from=1-1, to=1-2] 	\arrow[from=1-1, to=2-3] 	\arrow[hook, from=1-2, to=1-3] 	\arrow[from=1-2, to=2-3] 	\arrow[from=1-3, to=1-4] 	\arrow[from=1-3, to=2-3] 	\arrow[from=1-4, to=2-4] 	\arrow[hook, from=2-2, to=1-1] 	\arrow[from=2-2, to=2-3] 	\arrow[hook, from=2-3, to=2-4] \end{tikzcd}\]

By the Krull--Akizuki theorem, there exists a DVR $\Delta$ with
a map $\Spec\Delta\to\mathcal{W}_{\F_{q}}$ such that the generic
point of $\Delta$ maps to the generic point of $\mathcal{W}_{\F_{q}}$
isomorphically and the special point of $\Delta$ maps to the closed
point in $\mathcal{W}_{\F_{q}}$ corresponding to the Fermat hypersurface
of degree $d$. 

Let $\eta$ be the generic point of $\Delta$ (which is also the generic
point of $\mathcal{W}_{\F_{q}}$ by the above) and $s$ be the special
point of $\Delta$. Let $C_{\eta}$ be the generic fiber of $\mathcal{C}_{\mathcal{W}_{\F_{q}}}\to\Spec\mathcal{R}_{\F_{q}}$,
hence a smooth projective curve over $\eta$. Similarly, let $Y_{\eta},V_{\eta},$
and $X_{\eta}$ be the generic fibers of $\mathcal{Y}_{\mathcal{W}_{\F_{q}}}\to\Spec\mathcal{R}_{\F_{q}},\mathcal{V}_{\mathcal{W}_{\F_{q}}}\to\Spec\mathcal{R}_{\F_{q}},$
and $\mathcal{X}_{\mathcal{W}_{\F_{q}}}\to\Spec\mathcal{R}_{\F_{q}}$. 

We will take a suitable cover $C_{\eta}'\to C_{\eta}$ with $C_{\eta}'$
also a smooth projective curve that will be described in detail later. 

By the stable reduction theorem, after possibly replacing $\Delta$
with a dominating DVR, there exists a stable family $\mathcal{Z}\to\Spec\Delta$
so that the generic fiber $\mathcal{Z}_{\eta}$ is $C_{\eta}'$ and
the special fiber $\mathcal{Z}_{s}$ is a stable curve, both of the
same genus. 

We will then show that the composition $h'\colon\mathcal{Z}_{\eta}\to C_{\eta}\hookrightarrow Y_{\eta}\to V_{\eta}$,
which we may also view as a map $h'\colon\mathcal{Z}_{\eta}\to X_{\eta}$
by precomposition with the inclusion $V_{\eta}\hookrightarrow X_{\eta}$,
satisfies 
\[
\dim_{[h']}\Mor\left(\mathcal{Z}_{\eta},V_{\eta}\right)>\dim_{[h']}\Mor\left(\mathcal{Z}_{\eta},X_{\eta}\right),
\]
which is evidently a contradiction, since any map $\mathcal{Z}_{\eta}\to V_{\eta}$
can be extended to a map $\mathcal{Z}_{\eta}\to X_{\eta}$ via the
precomposition $V_{\eta}\hookrightarrow X_{\eta}$ uniquely. Here,
$\Mor\left(\mathcal{Z}_{\eta},V_{\eta}\right)$ and $\Mor\left(\mathcal{Z}_{\eta},X_{\eta}\right)$
are the schemes representing the spaces of morphisms $\mathcal{Z}_{\eta}\to V_{\eta}$
and $\mathcal{Z}_{\eta}\to X_{\eta}$, respectively.

Let us now explain how to construct the cover $C_{\eta}'\to C_{\eta}$.

Choose an integer $b\ge1$ such that 
\begin{equation}
ep^{b+1}\ge A\left(pg\left(C_{\eta}\right)+\frac{p-1}{2}\left(2g\left(C_{\eta}\right)-1\right)\right)+A\left(\frac{p-1}{2}\right).\label{eq:bchoice}
\end{equation}

Finally, we claim we can choose $m_{x}$ such that 

\begin{equation}
A\left(pg\left(C_{\eta}\right)+\frac{p-1}{2}\left(m_{x}-1\right)\right)\le ep^{b+1}<m\left(pg\left(C_{\eta}\right)+\frac{p-1}{2}\left(m_{x}-1\right)-1\right),\label{eq:mxchoice}
\end{equation}

\begin{equation}
m_{x}\ge2g\left(C_{\eta}\right),\label{eq:mxgenuscheck}
\end{equation}
and 

\begin{equation}
p\nmid m_{x}.\label{eq:mxdivisiblechoice}
\end{equation}

Note that increasing $m_{x}$ by 1 increases the quantity $A\left(pg\left(C_{\eta}\right)+\frac{p-1}{2}\left(m_{x}-1\right)\right)$
by $A\left(\frac{p-1}{2}\right)$, so let us choose $c$ so that $ep^{b+1}=A\left(pg\left(C_{\eta}\right)+\frac{p-1}{2}\left(c-1\right)\right)+E$
satisfies $A\left(\frac{p-1}{2}\right)\le E\le A\left(p-1\right)-1$.
In particular, note that $ep^{b+1}\ge A\left(pg\left(C_{\eta}\right)+\frac{p-1}{2}\left((c+1)-1\right)\right)$
as well. Moreover, by (\ref{eq:bchoice}), we have $c\ge2g\left(C_{\eta}\right)$.

Then, by (\ref{eq:pchoice}), we have 
\[
\left(\frac{1}{A}-\frac{1}{m}\right)ep^{b+1}\ge\left(\frac{1}{A}-\frac{1}{m}\right)p^{2}>p+\frac{p-1}{2}>\frac{E}{A}+1+\frac{p-1}{2},
\]

which can be rearranged to imply 
\[
ep^{b+1}<m\left(\frac{ep^{b+1}-E}{A}-1-\frac{p-1}{2}\right)=m\left(pg(C)+\frac{p-1}{2}\left(c-1\right)-1-\frac{p-1}{2}\right).
\]
If $p\nmid c$, let $m_{x}=c$; else, let $m_{x}=c+1$, which shows
that (\ref{eq:mxchoice}), (\ref{eq:mxgenuscheck}), and (\ref{eq:mxdivisiblechoice})
can be satisfied.

Consider $C_{\eta}'\to C_{\eta}$ a degree $p$ Artin--Schreier cover
given by $y^{p}-y=u$ with $u\in K\left(C_{\eta}\right)$ chosen so
that there is exactly one degree one point $x\in C_{\eta}$ for which
$v_{x}\left(u\right)=-m_{x}$ with $p\nmid m_{x}$ and $m_{x}$ a
positive integer, and $v_{x'}(u)\ge0$ for all other closed points
$x'$ by Riemann--Roch. More precisely, this follows from the fact
that $H^{0}\left(C_{\eta},\mathcal{O}\left(\left(m_{x}-1\right)x\right)\right)\to H^{0}\left(C_{\eta},\mathcal{O}\left(m_{x}x\right)\right)$
is not surjective, which is true because the former has dimension
$m_{x}-1-g+1$ and the latter has dimension $m_{x}-g+1$ by assumption
that $m_{x}\ge2g$.

Then, one can verify using Riemann--Hurwitz (see Proposition 3.7.8
of \cite{stichtenoth2009algebraic}) that 
\[
g\left(C_{\eta}'\right)=pg\left(C_{\eta}\right)+\frac{p-1}{2}\left(m_{x}-1\right).
\]

Moreover, consider the degree $p$ arithmetic Frobenius $C_{\eta}'\to C_{\eta}'$,
and iterate this $b$ times, so we obtain a degree $p^{b+1}$ map
$C_{\eta}'\to C_{\eta}$ given by first applying the Artin--Schreier
cover and then applying the Frobenius\textbf{ }$b$ times. Denote
by $h'$ the composition $C_{\eta}'\to C_{\eta}\to Y_{\eta}$. 

By the general lower bounds (as explained for example in Theorem 1.2
of \cite{Kollar_Rational_Curves}) on the dimension of moduli spaces
of maps, the fact that $Y_{\eta}\to V_{\eta}$ is birational, and
Lemma \ref{lem:lowerboundim-1} below, we have 
\[
\dim\Mor_{[h']}\left(\mathcal{Z}_{\eta},V_{\eta}\right)\ge\dim_{[h']}\Mor_{U}\left(\mathcal{Z}_{\eta},Y_{\eta}\right)\ge-h'^{*}K_{Y_{\eta}}\cdot\mathcal{Z}_{\eta}+\left(1-g\left(\mathcal{Z}_{\eta}\right)\right)\dim Y_{\eta},
\]
for some non-empty open subset $U$ of $Y_{\eta}$. 

Now, recalling the definition of the proper family $V_{d}\to U_{d}$
in Proposition \ref{cor:generalexpdim}, denote by $M\left(\mathcal{Z}_{\eta},X_{\eta}\right)$
the fiber above the hypersurface $X_{\eta}$ and similarly for $M\left(\mathcal{Z}_{s},X_{s}\right)$.

By inclusion for the first inequality, upper-semicontinuity of fiber
dimension for a proper family for the second, and Proposition \ref{cor:generalexpdim}
for the third, we have 
\begin{align*}
\dim\Mor\left(\mathcal{Z}_{\eta},X_{\eta}\right) & \le\dim M\left(\mathcal{Z}_{\eta},X_{\eta}\right)\le\dim M\left(\mathcal{Z}_{s},X_{s}\right)=-h'^{*}K_{X_{s}}|_{Y_{s}}\cdot\mathcal{Z}_{s}+\left(1-g\left(\mathcal{Z}_{s}\right)\right)(n-1).
\end{align*}
Combining these, we obtain
\begin{align*}
 & \dim\Mor_{[h']}\left(\mathcal{Z}_{\eta},V_{\eta}\right)-\dim\Mor_{[h']}\left(\mathcal{Z}_{\eta},X_{\eta}\right)\\
 & \ge-h'^{*}K_{Y_{\eta}}\cdot\mathcal{Z}_{\eta}+h'^{*}K_{X_{s}}|_{Y_{s}}\cdot\mathcal{Z}_{s}+\left(1-g\left(\mathcal{Z}_{\eta}\right)\right)\dim Y_{\eta}+\left(1-g\left(\mathcal{Z}_{s}\right)\right)\left(-n+1\right)\\
 & =-h'^{*}K_{Y_{\eta}}\cdot\mathcal{Z}_{\eta}+h'^{*}K_{X_{\eta}}|_{Y_{\eta}}\cdot\mathcal{Z}_{\eta}+\left(1-g\left(\mathcal{Z}_{\eta}\right)\right)\dim Y_{\eta}+\left(1-g\left(\mathcal{Z}_{\eta}\right)\right)\left(-n+1\right)\\
 & \ge\left(-h'^{*}K_{Y_{\eta}}+h'^{*}K_{X_{\eta}}|_{Y_{\eta}}\right)\cdot\mathcal{Z}_{\eta}+\left(g\left(\mathcal{Z}_{\eta}\right)-1\right)\\
 & >-\frac{ep^{b+1}}{m}+\left(g\left(C'\right)-1\right)\\
 & >0
\end{align*}
by the second inequality of (\ref{eq:mxchoice}), which yields the
necessary contradiction.
\end{proof}
\begin{lem}
\label{lem:lowerboundim-1} Let $C$ be a smooth projective curve
and $Y$ be a smooth variety. Suppose for every point $y$ of $Y$,
there exists a map $g\colon C\to Y$ with fixed curve class $g_{*}C=\alpha$
and whose image contains $y$. Let $\Mor\left(C,Y;\alpha\right)$
be the space of maps $C\to Y$ such that $f_{*}C=\alpha$ and $\Mor_{U}\left(C,Y;\alpha\right)$
be the subspace of such maps whose image intersects $U\subset Y$.
There is a non-empty open subset $U\subset Y$ and some map $f\colon C\to Y$
with curve class $f_{*}C=\alpha$ with the image of $C$ intersecting
$U$, such that $\dim_{[f]}\Mor_{U}\left(C,Y;\alpha\right)\ge-K_{Y}\cdot\alpha+\left(1-g(C)\right)n'$,
where $n'$ is the dimension of $Y$. 
\begin{proof}
For a fixed closed point $p\in C$, consider the morphism $\Mor\left(C,Y;\alpha\right)\to Y$
that sends a map to the image of $p$ in $Y$. We can similarly define
$\Mor\left(C,Y;\alpha,p\mapsto u\right)$ and $\Mor\left(C,Y;\alpha,p\mapsto U\right)$
to denote maps that send $p$ to a fixed point $u\in U$ and those
that send $p$ to some point in $U$, respectively. We then have the
following two Cartesian diagrams:

\[\begin{tikzcd}          {\Mor(C,Y;\alpha,p\mapsto u)} \ar[r] \ar[d] \ar[dr, phantom, "\square"] &  {\Mor(C,Y;\alpha,p\mapsto U)} \ar[d] \ar[dr, phantom, "\square"] \ar[r] & {\Mor(C,Y;\alpha)}\ar[d]  \\         u \ar[r]& U \ar[r] & Y     \end{tikzcd}\]Then,
for some non-empty open subset $U$, we have 
\begin{align*}
\dim_{[f]}\Mor\left(C,Y;\alpha,p\mapsto U\right) & =\dim U+\dim_{[f]}\Mor\left(C,Y;\alpha,p\mapsto u\right)\\
 & =n'+\dim_{[f]}\Mor\left(C,Y;\alpha,p\mapsto u\right).
\end{align*}
Using the lower bounds on $\dim_{[f]}\Mor\left(C,Y;\alpha,p\mapsto u\right)$
from Theorem 1.2 of \cite{Kollar_Rational_Curves}, we have 
\begin{align*}
\dim_{[f]}\Mor_{U}\left(C,Y;\alpha\right) & \ge\dim_{[f]}\Mor\left(C,Y;\alpha,p\mapsto U\right)\\
 & \ge n'-f^{*}K_{Y}\cdot C+\left(1-g\left(C\right)\right)n'-n'\\
 & =-K_{Y}\cdot\alpha+\left(1-g\left(C\right)\right)n'.
\end{align*}
\end{proof}
\end{lem}

\bibliographystyle{plainurl}
\bibliography{General_Fujita_Invariant_Actual}

\end{document}